\documentclass[a4 paper, 10pt]{article}
\usepackage{hyperref}
\usepackage{amssymb,dsfont,amsthm,amsmath,makeidx,verbatim,latexsym,amsfonts,mathrsfs}
\usepackage{cancel}
\usepackage{graphicx}
\usepackage[none]{hyphenat}

\usepackage[english]{babel}
\addto{\captionsenglish}{%

}
\DeclareMathAlphabet{\mathpzc}{OT1}{pzc}{m}{it}

\begin{document}
\hfuzz5pt
\theoremstyle{plain}
\newtheorem{theorem}{\textbf{Theorem}}[section]
\newtheorem{lemma}[theorem]{\textbf{Lemma}}
\newtheorem{proposition}[theorem]{\textbf{Proposition}}
\newtheorem{corollary}[theorem]{\textbf{Corollary}}
\newtheorem{claim}[theorem]{\textbf{Claim}}
\newtheorem{addendum}[theorem]{\textbf{Addendum}}
\newtheorem{definition}[theorem]{\textbf{Definition}}
\newtheorem{remark}[theorem]{\textbf{Remark}}
\newtheorem{example}[theorem]{\textbf{Example}}
\newtheorem{conjecture}[theorem]{\textbf{Conjecture}}
\newtheorem{notation}[theorem]{\textbf{Notation}}
\renewcommand{\baselinestretch}{1.50} 

	\pagenumbering{arabic} \baselineskip 10pt
\newcommand{\disp}{\displaystyle}
\thispagestyle{empty}

\title{Spherical Functions on Fuzzy Lie group}
\author{M. E. Egwe$^1$ and S.S. Sangodele$^2$\\ Department of Mathematics, University of Ibadan, Ibadan.\\ Email: $^1$\emph{murphy.egwe@ui.edu.ng},\;$^2$\emph{sangodele005@gmail.com}}
\maketitle

\begin{center}\textbf{Abstract}\end{center}
Let $G$ be a locally compact Lie group and $\mathfrak{g}$ its Lie algebra. We consider a fuzzy analogue of $G,$ denoted by $\mathfrak{G_f}$ called a fuzzy Lie group. Spherical functions on $\mathfrak{G_f}$ are constructed and a version of the existence result of the Helgason-spherical function on $G$ is then established on  $\mathfrak{G_f}.$
\section{Introduction}
\noindent
\par One of the basic tools for classical computation, modeling and reasoning is crisp, which is exact in nature.  A crisp is dichogamous, indicating Yes - or - No type rather than more - or - less - type. In this case, a membership function, is often used to  assign binary values to each element of the universal $X$.   Fuzzy set theory gives sufficient mathematical configuration in which vague conceptual facts can be precisely and rigorously examined. This has found application in many fields, including, computer science, biomedical engineering, telecommunication, decision making, differential equations, rings, semirings, group, automation and robotics, networking, discrete mathematics, e.t.c.
\par This originated from the novel work of  Zadeh in 1965 \cite{Zadeh} which was introduced to handle the notion of partial truth between "absolute true" and "absolute false".
Fuzzy vectors, fuzzy topological spaces were introduced and exhaustively considered  in \cite{Kastaras1},\cite{Kastaras2} and \cite{Kastaras3}. I. Kubiak \cite{Kubiak} and Sostak \cite{Sostak} considered the key idea of a fuzzy topological structures, as an expansion of both crisp and fuzzy topology. A locally convex property of these topologies has been given in \cite{Egwe2}
The idea of fuzzy topology on fuzzy sets was presented by Chalarabarty \& Ahsanullah \cite{Chakra} as one of the treatments of the issue  which might be known as the subspace issue in fuzzy topological spaces.
The general idea of fuzzy Lie algebras was introduced  by Akram in 2018 \cite{Akram} where fuzzy sets were applied to Lie algebras.
Nadja khah, M., et al.\cite{Nadja} gave the notion of fuzzy sets applications to Lie groups and  concepts relating to them. First, they considered $C^1$-fuzzy manifolds of a fuzzy transformation group and its fuzzy G-invariant property and then gave suitable conditions for defining fuzzy invariant differential operators on the $G.$
Spherical function on general locally compact groups  has been sufficiently studied (see \cite{Godement}, \cite{Chandra}, \cite{Egwe}, \cite{Yeol}).
Helgason, S. (1984), \cite{Helgason}  investigated the structure of the ring $DG(X)$ of $G$-invariant differential operators on a reductive spherical homogeneous space $X/G=H$ with an over group $D_eG$. We shall construct a polynomial algebra $D\mathfrak{G_f}^\mathfrak{K}(X)$ which is  $\mathfrak{G_f}$-invariant differential operators on $X$ with respect to $\mathfrak{K},$  coming from the centers of the enveloping algebra of $\mathscr{U}(G)$ of $\mathfrak{G_f}$ and $\mathfrak{K}$ where $\mathfrak{K}$ is a maximal proper subgroup of $\mathfrak{G_f}.$
\section{Preliminaries}
In this section,, we give some basis definitions that will be needed in the sequel following \cite{Akram}.
\begin{definition}\cite{Akram}
\normalfont
Let $S$ is a nonempty set and  $s\in S.$ A fuzzy set $\mu$ in a universe $S$ is a mapping $\mu:S\rightarrow [0,1]\subset\mathbb{R}.$ A fuzzy subset $\tilde{A}$ in $S$ is a set $\mu_{\tilde{A}}: S\longrightarrow [0,1]\subset\mathbb{R}$ also identified with the graph
$$\tilde{A} = \{(s,\mu_{\tilde{A}}(s))~\bigg|s\in S\},$$ of ordered pairs
where $\mu_{\tilde{A}}(s)$ is called the membership function.
\end{definition}
The fuzzy empty set is denoted by $\bigcirc_S$ and defined as $\bigcirc_S(s) = 0\;\forall\;s\in S$ while the entire set in a set $S$  is denoted $1_S$ and defined  as $1_S(s)= 1$ for all $s\in S.$
The basic operations on fuzzy sets and their standard results can be seen in \cite{Akram} and \cite{Yeol}

\begin{definition}
\normalfont
Let the universe of discourse be $S$ and $U$ a fuzzy set on a $S.$ Given $t\in [0,1],$  we  define a $t$-cut set ($t$-level set) of $U$  as
$$
U_t = \{s\in S| U(s) \geq t\}.
$$
\end{definition}

\subsection{Fuzzy Vector Spaces and Topology of Fuzzy sets}
\begin{definition}
\normalfont
Let $E$ be a vector space and $A_1\cdots A_n$ be fuzzy sets in $E$. We define $A_1\times\cdots \times A_n$ to be the fuzzy set $A$ in $E^n$ whose membership function is given by
$$\mu_A(s_1,\cdots, s_n) = \min\{\mu_{A_1}(s_1),\cdots,\mu_{A_n}(s_n)\}$$
Let $f:E^n\rightarrow E,~f(s_1,\cdots,s_n) = s_1+\cdots+s_n$. We define $A_1+\cdots +A_n = f(A)$.\\
For $\lambda\in\mathbb{K}$ and $D$ a fuzzy set in $E$, we define $\lambda D = g(D),$ where $g:E\rightarrow E,$ $g(s) = \lambda(s)$.
\end{definition}
\begin{definition}
\normalfont
Let $S\neq \emptyset$  and $I = [0,1].$ Define the set $I^S$ by $I^S:=\{\gamma|\gamma:S\rightarrow I\}.$ Then,  $\gamma\in I^S$ is known as fuzzy subset of $S$.
Let $\gamma\in I^S$ . A collection $\sigma$ of fuzzy subsets of $\gamma$ satisfying the following:
\begin{enumerate}
\item[(i)] $\kappa\cap \gamma \in \sigma~~\forall~ \kappa\in I$
\item[(ii)] $U_i\in\sigma~\forall~ i\in N\Rightarrow \bigcup\{U_i: i\in N\}\in \sigma$.
\item[(iii)] $U,V\in \sigma\Rightarrow U\cap V\in\sigma$
\end{enumerate}
is called a fuzzy topology on $\gamma$ and the pair $(\gamma,\sigma)$ is called a fuzzy topological space. Members of $\sigma$ are called fuzzy open sets and their complements with respect to $\gamma$ are known as closed sets of $(\gamma,\sigma)$.
\end{definition}
 If $\mathfrak{B}$ be a  collection of fuzzy subset of $\gamma$, then the family of arbitrary unions and finite intersections of the member of $\mathfrak{B}$ and the family  $\{\gamma\cap \kappa; ~\kappa\in I\}$  forms a fuzzy topology on $\gamma$ denoted by $\sigma(\mathfrak{B})$.

\begin{definition}
\normalfont
$\mathfrak{B}\in \sigma$ is referred to as open base of $\sigma$ if every member of $\sigma$ can be expressed uniquely as the union of certain members of $\mathfrak{B}$.
\end{definition}

\begin{definition}
\normalfont
A fuzzy topological space $(\gamma,\sigma)$ is said to be Hausdorff  if $\forall ~x_p, y_p\in \gamma\;(x\neq y), ~\exists U,V\in\sigma$ such that $x_p\in U,
y_p\in V$ and $U\cap V =\emptyset$
\end{definition}

\begin{definition}
\normalfont
A fuzzy topological space $(\gamma,\sigma)$ is said to be fuzzy compact if $\forall ~\beta\in\sigma$ with $\bigcup\{U: U\in\beta\} = \gamma$ and $\forall \varepsilon>0,~\exists$ a finite subcollection $\beta_0$ of $\beta$ such that $\bigcup\{U, U\in\beta_0\}\geq \gamma_\varepsilon$ where $\gamma_\varepsilon$ is defined by $\gamma_\varepsilon(x) = \gamma(x)-\varepsilon$ or $0$ according as $\gamma(x)>\varepsilon$ or $\gamma(x)\leq \varepsilon$
\end{definition}

\begin{definition}
\normalfont
A fuzzy subset $U$ of $\gamma$ is called fuzzy separated if $\exists~~v,\delta\in\sigma$ such that $U = v\cup \delta,~~v\neq \delta$ and $v\cap \delta = \emptyset$.
\end{definition}

\begin{definition}
\normalfont
A fuzzy topological space $(\gamma,\sigma)$  is said to be  connected in the fuzzy sense if for any $\beta_0$ fuzzy closed subset of $( \gamma,\sigma)$ can be fuzzy separated.\\
\end{definition}
We shall take for granted that all information and ideas needed on fuzzy Lie algebras follow from \cite{Akram}.
\begin{definition}
\normalfont
Let $W$ be a vector space over  $\mathbb{K}$. A fuzzy subset $U$ of $W$ satisfying the following conditions
\begin{itemize}
\item[(i)] $U(s+t)\geq \min\{U(s), U(t)\}$~ for all $s,t\in W$
\item[(ii)] $U(\alpha s)\geq U(s)$ for all $s\in W,~\alpha\in \mathbb{K}$
\end{itemize}
is called a fuzzy subspace of $W.$
\end{definition}

\begin{definition}
\normalfont
A fuzzy set $U:{\L}\longrightarrow [0,1]$ is called a fuzzy Lie subalgebra of ${\L}$ over a field $\mathbb{K}$ if it is a fuzzy subspace of ${\L}$ such that
\begin{enumerate}
\item[(i)] Each non empty $U(L,t)$ is a subspace of ${\L}$
\item[(ii)] $U([s,t])\geq \min\{{\L}(s), U(t)\}$.
\end{enumerate}
hold for all $s,t\in {\L}$ and $\alpha\in\mathbb{K}$
\end{definition}

\begin{definition}
\normalfont
A fuzzy set $U:{\L}\rightarrow [0,1]$ is called a fuzzy Lie ideal of ${\L}$ if
\begin{enumerate}
\item[(i)] $U(s+t)\geq \min\{U(s), U(t)\}$
\item[(ii)] $U(\alpha s)\geq U(s)$
\item[(iii)] $U([s,t])\geq U(s)$
\end{enumerate}
hold for all $s,t\in {\L}$ and $\alpha\in\mathbb{K}$
\end{definition}
\begin{definition}
\normalfont
Let ${\L}_1$ and ${\L}_2$ be two Lie algebras and $\varphi$ a function from ${\L}_1$ to ${\L}_2$. If $U$ is a fuzzy set in ${\L}_2$, then the pre-image of $U$ under $\varphi$ is the fuzzy set in ${\L}_1$ defined by
$$\varphi^{-1}(U)(s) = U(\varphi(s))~~\forall~~ s\in {\L}_1$$
\end{definition}

\textbf{Example 2.2:}
Let ${\L}$ be a Lie algebra, such that ${\L} = {\L}_0+{\L}_1$ and ${\L}_0 = sl(2,\mathbb{R}),~{\L}_1 = 0$. For any $x,y\in {\L}_0$, then $[x,y] = x\times y$ where $\times$ is the cross-product.\\
We show that
$$
U: {\L}\longrightarrow [0,1]~~\mbox{  is a fuzzy Lie subalgebra}
$$
\subsection*{\underline{Solution}}
We define
$U_{\bar{0}}: {\L}_{\bar{0}}\longrightarrow [0,1]$ as
\begin{equation*}
\left\{
\begin{split}
&1,~\mbox{  if  }~ a = b = c = 0\\
&0.5,~\mbox{  if  }~ a \neq 0 ~\mbox{  and  }~ b = c = 0\\
&0,~\mbox{  if  }~ \mbox{  Otherwise  }
\end{split}
\right.
\end{equation*}
where
$$
A = \begin{pmatrix}
a&&b\\
c&&-a
\end{pmatrix}\in sl(2,\mathbb{R})
$$
Also define
$U_{\overline{1}}: {\L}_{\overline{1}}\longrightarrow [0,1]$ as $U_{\overline{1}} = 1.$ By extension,
\begin{equation*}
\begin{split}
U^\prime_{\bar{0}}:{\L}&\longrightarrow [0,1]~\mbox{  is  }\\
U_{\bar{0}}^\prime(x)& =
\left\{
\begin{split}
&U_{\bar{0}}(x),~~x\in {\L}_{\bar{0}}\\
&0,~~~~x\notin {\L}_{\bar{0}}
\end{split}
\right.
\end{split}
\end{equation*}
also

\begin{equation*}
\begin{split}
U^\prime_{\bar{1}}{\L}&\longrightarrow [0,1]~\mbox{  is  }\\
U_{\bar{1}}^\prime(x)& =
\left\{
\begin{split}
&U_{\bar{1}}(x),~~x\in {\L}_{\bar{1}}\\
&0,~~~~x\notin {\L}_{\bar{1}}
\end{split}
\right.
\end{split}
\end{equation*}
Then
$$
U:{\L}\longrightarrow[0,1] 
$$
defined by
\begin{align*}
U(x)& = U^\prime_{\bar{0}}(x)\\
U& = U_{\bar{0}}\oplus U_{\bar{1}}
\end{align*}
\begin{align*}
(U^\prime_{\bar{0}}\oplus U^\prime_{\bar{1}})(x)& = \sup_{x = a+b}\{\min(U^\prime_{\bar{1}}(a), U^\prime_{\bar{1}}(b))\}\\
& = \min(U^\prime_{\bar{0}}(x_{\bar{0}}), U^\prime_{\bar{1}}(x_{\bar{1}})) = U(x)
\end{align*}
and
$$
U^\prime_{\bar{0}}\cap U^\prime_{\bar{1}} = 1_{\bar{0}}
$$
Hence, $U$ is a fuzzy subalgebra. However, $U$ is not a fuzzy ideal of ${\L}$ because
\begin{align*}
&U([(1,0,0), (1,1,1)]) = U^\prime_{\bar{0}}([(1,0,0),(1,1,1)])\\
& = U^\prime_{\bar{0}}((1-1), -(1-0), (1-0))\\
& = U^\prime_{\bar{0}}((1-1), -(1-0), (1-0))\\
& = U^\prime_{\bar{0}}(0,-1,1) = 0
\end{align*}
However
\begin{align*}
&\max\left\{U(1,0,0), U(1,1,1) = \right.\\
&\max\left\{ U_{\bar{0}}(1,0,0), U_{\bar{0}}(1,1,1)=\right.\\
&\max\left\{0.5,0] = 0.5\right.
\end{align*}
such that
$$
U([x,y])\ngeq\max\{U(x), U(y)\}
$$
\subsection{Fuzzy Topological Groups}
\noindent
\par We now introduce notion of fuzzy topological group and its corresponding differentiable manifold in what follows.
\begin{definition}\label{2.6}
\normalfont
A fuzzy subset $F$ of $X\times Y$ is said to be a fuzzy proper function from $\lambda$ to $U$ if
\begin{enumerate}
\item[(i)] $F(x,y)\leq \min\{\lambda(x), U(y)\}~~\forall~~ (x,y) \in X\times Y$
\item[(ii)]$\forall ~x\in X, ~\exists~y_0\in Y$ such that $F(x,y_0) = \lambda(x)$ and $F(x,y) = 0$ if $y\neq y_0$
\end{enumerate}
\end{definition}
\begin{definition}\label{3.1}
\normalfont
A proper function $F:(\lambda,\tau)\longrightarrow (U,\sigma)$, is said to be
\begin{enumerate}
\item[(i)] Fuzzy continuous is $F^{-1}(U)\in\tau,~~\forall ~ U\in \sigma,$
\item[(ii)] Fuzzy open if $F(\delta)\in\sigma~\forall~ \delta\in\tau$
\item[(iii)] Fuzzy homeomorphism if $F$ be bijective, fuzzy continuous and open.
\end{enumerate}
\end{definition}
\begin{definition}\cite{Nadja}
\normalfont
A fuzzy topology $\sigma$ on a group $G$ is said to be compatible if the mappings
\begin{align*}
m&: G\times G, \sigma\times \sigma)\longrightarrow (G,\sigma):m(x,y)\mapsto xy\\
i&:(G,\sigma)\longrightarrow (G,\sigma):i(x)\mapsto x
\end{align*}
are fuzzy continuous. A group $G$ equipped with a compatible fuzzy topology $\sigma$ on $G$ is called a fuzzy topological group.
\end{definition}
\begin{definition}\label{4.6}
\normalfont
A fuzzy topological space is  $T_1$ space if, and only if every fuzzy point in it is a closed fuzzy set.
\end{definition}
\begin{definition}
\normalfont
A fuzzy topological vector space (ftvs) is a vector  space  $\Xi$ over the field $\mathbb{K},\;\;\Xi$ equipped with a fuzzy topology $\sigma$ and $\mathbb{K}$ equipped with the usual topology $\kappa$, such that the two mappings
\begin{itemize}
\item[(i)] $(s,t)\longrightarrow s+t~\mbox{of}~ (\Xi,\sigma)\times(\Xi,\sigma)$ into $(\Xi,\sigma)$
\item[(ii)] $(\alpha,s)\longrightarrow \alpha s~\mbox{of}~ \mathbb{(K},\kappa)\times(\Xi,\sigma)$ into $(\Xi,\sigma)$
\end{itemize}
are fuzzy continuous.
\end{definition}
\begin{definition}\cite{Nadja}
\normalfont
Let $\Xi,\digamma$ be two fuzzy topological vector spaces. The mapping $\phi:\Xi\rightarrow\digamma$ is said to be tangent at $0$ if given a neighbourhood $W$ of $0_\delta,\;0<\delta<1$ in $\digamma,$ there exists a neighbourhood $V$ of $0_\epsilon\;0<\epsilon<\delta$ in $\Xi$  such that $$\phi[tV]\subset \o(t)W$$
for some function $\o(t)$.
\end{definition}

\begin{definition}
\normalfont
Let $\Xi$ and $\digamma$ be two fuzzy topological vector space, each equipped with a fuzzy topology (possibly $T_1$). Let $\varphi:\Xi\longrightarrow \digamma$ be a fuzzy continuous mapping. Then, $\varphi$ is called fuzzy differentiable at  $s\in \Xi$ if $\exists$ a continuous linear fuzzy map $U:\Xi\rightarrow\digamma$  such that
$$\varphi(s+t) = \varphi(s)+U(t)+\varrho(t), ~~~t\in \Xi.$$
\end{definition}
where $\varrho$ is tangent to 0. This mapping $U$ is known as the fuzzy derivative  of $\varphi$ at $s_\circ$. The fuzzy derivative of $\varphi$ at $s_\circ$ is denoted by $\varphi^\prime(s_\circ)$; it is an element of $L(\Xi,\digamma):=\{\varphi|\varphi:\Xi\rightarrow\digamma,\;\varphi\;\mbox{is fuzzy linear and continuous}\}.$  $\varphi$ is differentiable in the fuzzy sense if it is  differentiable at every point in $\Xi$ in the fuzzy sense.

\begin{definition}
\normalfont
Let $\Xi,\digamma$ be ftvs. A bijection $\varphi:\Xi\rightarrow\digamma$ is called a fuzzy diffeomorphism of class $C^1$ if $\varphi$ and its inverse $\varphi^{-1}$ are differentiable in the fuzzy sense and $\varphi^\prime$ and $(\varphi^{-1})^\prime$ are fuzzy continuous.
\end{definition}
The whole idea of fuzzy $C^1$-manifold and atlas is well-known (see \cite{Nadja}).
\begin{definition}
\noindent
A fuzzy Lie group, $\mathfrak{G_f}$ is a $C^1$-fuzzy manifold $\mathfrak{G_f}$ which is also a group, such that the mappings
\begin{align*}
m&: (\mathfrak{G_f}\times \mathfrak{G_f}, \sigma\times \sigma)\longrightarrow (\mathfrak{G_f},\sigma)\\
i&:(\mathfrak{G_f},\sigma)\longrightarrow (\mathfrak{G_f},\sigma)
\end{align*}
are fuzzy differentiable.
\end{definition}
\section{Spherical Functions on Fuzzy Lie Groups}
In this section, we first introduce the general construction of the popular spherical functions as can be seen in \cite{Helgason}, \cite{Godement}, \cite{Egwe1} etc.
\begin{definition}
\normalfont
\noindent
Let $\mathfrak{G_f}$ be a fuzzy Lie group and $\mathfrak{K_f}$ a closed fuzzy Lie subgroup of $\mathfrak{G_f}.$ Suppose $\varrho$ be a complex-valued function on $\mathfrak{G_f}/\mathfrak{K_f}$ (where $\mathfrak{G_f}$ is a fuzzy Lie group and $\mathfrak{K_f}$ is a compact subgroup) of class $C^\infty$ which satisfies $\varrho(\pi(e)) = 1.$ Then, $\varrho$ is referred as spherical function if
\begin{enumerate}
\item[(i)] $\varrho^{\sigma(k)} = \varrho\;\forall k\in \mathfrak{K_f}$\\
\item[(ii)] $D_\varrho = \lambda_D\varrho$ for each $D\in D(\mathfrak{G_f}/\mathfrak{K_f}).$
\end{enumerate}
Here, $\lambda_D$ is a complex number and $D(\mathfrak{G_f}/\mathfrak{K_f})$ is the algebra of a differential operators on $\mathfrak{G_f}/\mathfrak{K_f}$ invariant under all the translations
$$
\sigma(g): x\mathfrak{K_f}\rightarrow gx\mathfrak{K_f}~~\mbox{  of  }~~ \mathfrak{G_f}/\mathfrak{K_f}\eqno(T)
$$
\end{definition}
In what follows, we shall construct spherical function on a locally compact fuzzy Lie group $\mathfrak{G_f}$ following the \cite{Helgason}.
\noindent
\par Let $M$ be a $C^1-$ fuzzy manifold and $\varrho: M\longrightarrow M$ a $C^1-$ fuzzy diffeomorphism of $M$ onto itself. We put
$$
f^{\varrho} = f\circ \varrho^{-1},~~f\in \varepsilon(M)
$$
and if $D$ is a differential operator on $M$, we define $D^\varrho$ by
$$
D^\varrho: f\longrightarrow (Df^{\varrho-1})^\varrho = (D(f\circ \varrho))\circ \varrho^{-1},~~f\in \varepsilon(M)
$$
where $D^\varrho$ is another differential operator. The operator $D$ is said to be invariant under $\varrho$ if $D^\varrho = D$ i.e. $D(f\circ \varrho) = Df\circ \varrho$ for all $f$. Note that $(Df)^\varrho = D^\varrho f^\varrho$.\\
If $T$ is a distribution on $M$, we put $T^\varrho$ for the distribution $T^\varrho(f) = T(f^{\varrho_{-1}}),~f\in D(M)$.\\

Let $\mathfrak{G_f}$ be a fuzzy Lie group, $\mathfrak{H_f}\subset \mathfrak{G_f}$ a closed fuzzy subgroup of $\mathfrak{G_f}$. $\mathfrak{G_f}/\mathfrak{H_f}$ the fuzzy manifold of left cosets $g\mathfrak{H_f}(g\in \mathfrak{G_f})$ and $D(\mathfrak{G_f}/\mathfrak{H_f})$ the algebra of all differential operators on $\mathfrak{G_f}/\mathfrak{H_f}$ which are invariant under the usual transformations $(T)$

Given a coset space $\mathfrak{G_f}/\mathfrak{H_f}$, we intend to define the operator in $D(\mathfrak{G_f}/\mathfrak{H_f})$. We consider a case when $\mathfrak{H_f}=\{e\}$ and put $D(\mathfrak{G_f})$ for $D(\mathfrak{G_f}/\{e\})$, the set of left-invariant differential operators on $\mathfrak{G_f}$.\\
\par If $\bigsqcup$ is a fuzzy vector spaces over $\mathbb{R}$, the symmetric fuzzy algebra $\mathcal{S}(\bigsqcup)$ over $\bigsqcup$ is defined as the fuzzy algebra of complex-valued polynomial functions on the dual space $\bigsqcup^\star$. If $X_1,\ldots,X_n$ is the basis for $\bigsqcup,$ then $\mathcal{S}(\bigsqcup)$ can be identified with the commutative fuzzy algebra of polynomials
$\displaystyle\sum_{(k)}a_{k_1,\ldots,k_n}X_1^{k_1}\cdots X_n^{k_n}.$
Let $\mathfrak{g}$ denote the fuzzy Lie algebra of $\mathfrak{G_f}$ (the tangent space to $\mathfrak{G_f}$ at $e$) and $\exp:\mathfrak{g}\mapsto \mathfrak{G_f}$ the exponential mapping which maps a line through $0$ in $\mathfrak{g}$ onto a one parameter subgroup $t\longrightarrow \exp tx$ of $\mathfrak{G_f}$. If $X\in \mathfrak{g},$ let $\tilde{X}$ denote the fuzzy vector field on $\mathfrak{g}$ given by
\begin{equation}\label{(1)}
(\tilde{X}f)(g) = X(f\circ L_g) = \bigg\{\frac{d}{dt}f(g\exp tX)\bigg\}_{t = 0} ~~\mbox{  for  }~~ f\in\mathfrak{E}(\mathfrak{G_f}).
\end{equation} where $L_g$ denote the left translation $x\longrightarrow gx$ of $\mathfrak{G_f}$ onto itself. Then $\tilde{X}$ is a differential operator on $\mathfrak{G_f}$ if $h\in \mathfrak{G_f}$ then
$$
(\tilde{X}^{L_h}f)(g) = (\tilde{X}(f\circ L_h))(h^{-1}g) = (\tilde{X}f)(g)
$$
So $\tilde{X}\in D(\mathfrak{G_f})$. Moreover, the bracket on $\mathfrak{g}$ is by definition given by
$$
[X,Y]^\sim = \tilde{X}\tilde{Y} - \tilde{Y}\tilde{X},~~ X,Y\in \mathfrak{g}
$$
the multiplication on the right-hand side being composition of operators.
\begin{definition}\cite{olafsson}
The pair $(\mathfrak{G_f},\mathfrak{K})$ is called a symmetric pair if there exists a involutive automorphism $\theta$ of $\mathfrak{G_f}$ such that $\mathfrak{G_f}_\theta^0\subset \mathfrak{K}\subset \mathfrak{G_f}_\theta,$ where $\mathfrak{G_f}$ is the set of fixed points of $\theta$ and $\mathfrak{G_f}_\theta^0$ is the identity component. The space $\mathfrak{G_f}/\mathfrak{K}$ is called a symmetric space.
\end{definition}
\begin{theorem}\label{Theorem1}
\normalfont
Let $\mathfrak{G_f}$ be  fuzzy Lie group with fuzzy algebra
$\mathfrak{g}$. Let $\mathcal{S}(\mathfrak{g})$ denote the symmetric fuzzy algebra over the fuzzy vector space $\mathfrak{g}$. Then there exists a unique linear bijection
$$
\lambda: \mathcal{S}(\mathfrak{g})\longrightarrow D(\mathfrak{g})
$$
such that
\begin{equation}\label{(2)}
\lambda(X^m) = \tilde{X}^m,\;\; X\in \mathfrak{g}, ~m\in [0,1]
\end{equation}
If $X_1\cdots X_n$ is any basis of $\mathfrak{g}$ and $p\in S(\mathfrak{g})$, then
\begin{equation}\label{(3)}
(\lambda(p) f)(g) = \{p(\partial_1,\ldots,\partial_n)f(g\exp(t_1 X_1+\cdots+t_nX_n))\}_{t  = 0},
\end{equation}
where $f\in \mathfrak{E}(\mathfrak{G_f}),\partial_i = \frac{\partial}{\partial t_i}$ and $t = (t_1,\cdots, t_n).$
\end{theorem}
\begin{proof}
for any fixed basis $X_i,\cdots, X_n$ of $\mathfrak{g}$. The mapping
$$
g\exp(t_1 X_1+\cdots+t_nX_n)\longrightarrow (t_1,\cdots,t_n)
$$
is a coordinate system on a neighbourhood of $g$ in $\mathfrak{G_f}$. By equation \eqref{(2)}, define a differential operator $\lambda(P)$ on $\mathfrak{G_f}$. Clearly $\lambda(p)$ is left invariant, and by \eqref{(1)} $\lambda(X_i) = \tilde{X}_i$, so by linearity $\lambda(X) = \tilde{X}$ for $X\in \mathfrak{g}$.\\
Also we show that $\lambda$ is one - to - one.\\
Suppose $\lambda(P) = 0$ where $P\neq 0$ with respect to a lexicographic ordering. Let $a X_1^m\cdots X_n^m$ be the leading term in $P$. Let $f$ be a smooth function on a neighbourhood of $e$ in $\mathfrak{G_f}$ such that
$$
f(\exp(t_1X_1+\cdots+t_nX_n))^n = t_1^{m_1}\cdots t_n^{m_n}
$$
for small $t$, then $(\lambda(p)f)(e)\neq 0$ contradicting $\lambda(P) = 0$.\\
Finally, $\lambda$ maps $S(\mathfrak{g})$ onto $D(\mathfrak{g})$. Also if $u\in D(\mathfrak{g})$, there exist a polynomial $P$ such that
$$
(uf)(e) = \{P(\partial_1,\ldots,\partial_n)+ (\exp t_1 X_1+\cdots + t_n X_n)\}.
$$
Then by the left invariance of $u,$ $u = \lambda(P)$ so $\lambda$ is surjective.
\end{proof}
The mapping $\lambda$ is usually called symmetrization and it has the following properties.\\
If $a_1,\ldots, a_p\in \mathfrak{g}$ then
\begin{equation}
\lambda(a_1,\ldots, a_p) = \frac{1}{p!}\sum_{\sigma\in \mathfrak{S}p}{a}_{\sigma(1)}\cdots {a}_{\sigma(p)},
\end{equation}
where $\mathfrak{S}p$ is the symmetric group on $p$ letters and we have $(t_1a_1+\cdots+t_pa_p)^p$ by equating the coefficients to $t_1,\cdots, t_p$. We can recall some facts concerning the adjoint representation  $Ad\mathfrak{g}$.Let $g\in \mathfrak{G_f}$, the mapping $x\rightarrow gxg^{-1}$ is an automorphism of $\mathfrak{G_f}$, and the corresponding automorphism of $	\mathfrak{g}$ is denoted by $Ad(g)$\\
Thus,
\begin{equation}\label{(5)}
\exp Ad(g)X = g\exp Xg^{-1},~~ X\in \mathfrak{g}, ~~g\in \mathfrak{G_f}.
\end{equation}
Then the mapping $g\longrightarrow Ad(g)$ is a representation of $\mathfrak{G_f}$  and this induces a representation of $\mathfrak{g}$ on $\mathfrak{g}$ denoted by $ad$
\begin{equation}\label{(6)}
Ad(\exp X) = e^{adX},~~X\in \mathfrak{g}
\end{equation}
where, for a linear transformation $A$, $e^A$ denotes $\displaystyle\sum_0^\infty(1/n!)A^n.$ From \eqref{(5)} and \eqref{(6)} we can say,
\begin{equation}\label{(7)}
ad X(Y) = [X,Y].~~~X,Y\in \mathfrak{g}
\end{equation}
We can now extend this operation to differential operators .\\
\par We calculate $(Ad(gX)^\sim$. To do this, recall the translation
$$
L_g:x\longrightarrow gx,~~R_g:x\longrightarrow xg
$$
and for any $f\in\mathfrak{E}(\mathfrak{G_f})$
\begin{align*}
\bigg[(Ad(gX)^\sim f\bigg](x) & = \bigg\{\frac{d}{dt}f(x\exp t Ad(g)X)\bigg\}_{t = 0}\\
 =\bigg\{\frac{d}{dt}f(xg\exp t X g^{-1})\bigg\}& = \bigg\{\frac{d}{dt}f^{R_g}(xg\exp tX)\bigg\}_{t = 0}\\
=\bigg(\tilde{X}f^{R_g}\bigg)(xg& = \bigg(\tilde{X}f^{R_g}\bigg)^{R_{g^{-1}}}(x)
\end{align*}
and
$$
\bigg(Ad(g)X\bigg)^\sim = \tilde{X}^{R_{g^{-1}}}.
$$
Thus we define for $D\in D(\mathfrak{G_f})$
\begin{equation}\label{(8)}
Ad(g)D = D^{R_{g^{-1}}}.
\end{equation}
Hence $Ad(g$ is an automorphism of $D(\mathfrak{G_f})$. Also we observe that
$$
(ad(X)(Y))^\sim = \tilde{X}\tilde{Y} - \tilde{Y}\tilde{X},
$$
hence we define for $D\in D(\mathfrak{G_f})$

\begin{equation}\label{(9)}
(ad~X)(D) = \tilde{X}D - D\tilde{X},
\end{equation}
and then $adX$ is a derivative of the algebra $D(\mathfrak{G_f})$.\\
We define
\begin{equation}\label{(10)}
e^{adX}(D) = \sum_{0}^\infty\frac{1}{n!}(adX)^n(D),~~ D\in D(\mathfrak{G_f})
\end{equation}
 because $(adX)^n(D)$ by \eqref{(9)} is a differential operator of order $\leq $ order of $D$, thus all the terms in the series \eqref{(10)} lie in a fuzzy vector space, so there is no convergence problem. Also applying Leibniz's formula for the power of a derivation applied to a product, we have for $D_1, D_2\in D(\mathfrak{G_f})$
\begin{align*}
e^{adX}(D_1D_2) = \sum_0^\infty\frac{1}{n!}(ad X)^n(D_1D_2) &= \sum_0^\infty\frac{1}{n!}\sum_{0\leq i\cdot j,~i+j = n}\frac{n!}{i!j!}(adX)^i(D_1)(adX)^j(D_2)\\
& =  e^{adX}(D_1D_2) = e^{adX}(D_1)e^{adX}(D_2)
\end{align*}
Thus $Ad(\exp X)$ and $e^{adX}$ are automorphism of $D(\mathfrak{G_f})$, they coincide on $\tilde{g}$, hence on all on $D(\mathfrak{G_f})$ since  by \cite{Helgason}, $\mathfrak{g}$ generates $D(\mathfrak{G_f})$, consequently,
$$Ad(\exp X)D = e^{adX}(D) = e^{adX}(D),~~D\in D(\mathfrak{G_f})$$
\section{The main reuslts}
\noindent
\par Let $\mathfrak{G_f}$ be a fuzzy Lie group, and $\mathfrak{K}$ a closed fuzzy subgroup of $\mathfrak{G_f}$. Let $\varpi$ be the natural mapping of $\mathfrak{G_f}$ onto $\Xi = \mathfrak{G_f}/\mathfrak{K}$. Let $0=\varpi(e)$ and $\widetilde{f}=f\circ\varpi$ for $f$  any function on $\Xi$. We denote by $D(\mathfrak{G_f})$  the set of all (left-) invariant differential operators on $\mathfrak{G_f}$, and $D_\mathfrak{K}(\mathfrak{G_f})$ the subspace of all right invariant differential operators under $\mathfrak{K}$, then clearly,  $D(\mathfrak{G_f}/\mathfrak{K})$ denotes the algebra of differential operators on $\mathfrak{G_f}/\mathfrak{K}$ invariant under the usual translations.
We shall prove the existence of the following result in the fuzzy sense.
\begin{theorem}\cite{Helgason}
\normalfont
Let $f$ be a complex-valued continuous function on $\mathfrak{G_f}$, not identically $0$. Then $f$ is a spherical function if and only if

\begin{equation}
\int_Kf(xky) = f(x)f(y)~~\forall~ x,y\in \mathfrak{G_f}
\end{equation}
\end{theorem}
\begin{theorem}\cite{Helgason}
\normalfont
Let $C_c^*(\mathfrak{G_f})$ be the space of nonzero continuous functions on  with compact support on $\mathfrak{G_f}.$  Let $\varrho$ be a continuous complex-valued function on $\mathfrak{G_f}$ bi-invariant under $\mathfrak{K}$. Then $\varrho$ is a spherical function if and only if the mapping
$$
L:f\rightarrow \int_\mathfrak{G_f}f(x)\varrho(x)dx
$$
is a homomorphism of $C_c^*(\mathfrak{G_f})$ onto $\mathbb{C}$.
\end{theorem}

\begin{theorem}
\normalfont
The algebra $D(\mathfrak{G_f}/\mathfrak{K})$ is commutative.
\end{theorem}
\begin{proof}
Let $\mathfrak{G_f}$ be a fuzzy Lie group and $\mathfrak{g}$ be a fuzzy Lie algebra of $\mathfrak{G_f}$. Let $\mathfrak{K}$ be a closed fuzzy subgroup of $\mathfrak{G_f}$ and the symmetric space $\mathfrak{M}= \mathfrak{G_f}/\mathfrak{K}$ be a fuzzy manifold of left cosets $gk\;g\in \mathfrak{G_f}$ where $D(\mathfrak{G_f}/\mathfrak{K})$ the algebra of all differential operators on $\mathfrak{G_f}/\mathfrak{K}$ which are invariant under the usual transformations.
Let $f\in C^1(\mathfrak{M})$ and let $x_1,x_2,\ldots,x_n$ be a basis in $\mathfrak{M}$. Let $\tilde{X}\in \mathfrak{g}$ and $\tilde{Y}\in \mathfrak{g}$, we define
$$
\tilde{X} = \sum_{j=1}^n U_j\frac{\partial}{\partial x_j}
$$
and
$$
\tilde{Y} = \sum_{i=1}^n V_i\frac{\partial}{\partial x_i}
$$
we have
$$
(\tilde{X}\circ \tilde{Y})f = \tilde{X}(\tilde{Y}f) = \tilde{X}\left(\sum_{i = 1}^n V_i\frac{\partial f}{\partial x_i}\right)
$$

\begin{equation}\label{E1}
 = \sum_{i,j = 1}^n U_j \frac{\partial V_i}{\partial x_j}\frac{\partial f}{\partial x_i} + \sum_{i,j = 1}^n U_j V_i\cdot\frac{\partial^2f}{\partial x_j.}
\end{equation}
Also
$$
(\tilde{Y}\circ \tilde{X})f = \tilde{Y}(\tilde{X}f) = \tilde{Y}\left(\sum_{j= 1}^n V_j\frac{\partial f}{\partial x_j}\right)
$$
\begin{equation}\label{E2}
 = \sum_{i,j = 1}^n V_i\frac{\partial U_j}{\partial x_i}\frac{\partial f}{\partial x_j} + \sum_{i,j = 1}^n V_iU_j \frac{\partial^2 f}{\partial x_j \partial x_i.}
\end{equation}
By subtraction, we obtain for $f\in C^2(\mathfrak{M})$
\begin{equation}\label{E3}
\begin{array}{rcl}
\tilde{X}(\tilde{Y}f) - \tilde{Y}(\tilde{X}f) &=& \displaystyle\sum_{i = 1}^n\left(\sum_{j = 1}^n\left(U_j\frac{\partial V_i}{\partial x_j} - V_i\frac{\partial U_j}{\partial x_j}\right)\right)\frac{\partial f}{\partial x_i}\\
 &=& \displaystyle\sum_{i,j = 1}^n U_j\frac{\partial V_i}{\partial x_j}\frac{\partial f}{\partial x_i} - \sum_{i,j = 1}^n V_i\frac{\partial U_j}{\partial x_j}\frac{\partial f}{\partial x_i}\\
 &=& 0, \;\mbox{}.
\end{array}
\end{equation}
This implies that
$\tilde{X}(\tilde{Y}f) - \tilde{Y}(\tilde{X}f) = 0,$\\
$\tilde{X}(\tilde{Y}f)  =  \tilde{Y}(\tilde{X}f)$ whenever $U=kV$ for $k$ any scalar.
\end{proof}

We are now ready to use \eqref{E2} and \eqref{E3}  to define spherical function.\\
\noindent

\par Let $\mathfrak{G_f}$ be a fuzzy Lie group and $\mathfrak{K}$ a maximal compact subgroup. Let $C_c(\mathfrak{K}\setminus \mathfrak{G_f}/\mathfrak{K})$ denote the space of all continuous functions with compact support on $\mathfrak{G_f}$ which satisfy $f(k_1gk_2)=f(g)$ for all $k_1,k_2\in \mathfrak{K}.$ Such spaces are called spherical or bi-invariant. Then, $C_c(\mathfrak{K}\setminus \mathfrak{G_f}/\mathfrak{K})$ forms a commutative Banach algebra under convolution \cite{Egwe1} and we call the pair $(\mathfrak{G_f},\mathfrak{K})$ a Gelfand pair.
\par Let $\mathfrak{G_f}$ be a fuzzy Lie group and $\mathfrak{K}$ a closed fuzzy subgroup of $\mathfrak{G_f}$. Let $\mathfrak{M} = \mathfrak{G_f}/\mathfrak{K}$ be a symmetric space. For any function $\varrho\in \mathfrak{G_f}~\exists$ a function $\varrho\in C_c(\mathfrak{K}\setminus \mathfrak{G_f}/\mathfrak{K})$ which satisfies $\varrho(kgk^\prime) = \varrho(g)$ and are integrable on $\mathfrak{G_f}$  for a normed algebra $A$ under the convolution product of $\mathfrak{G_f},\;g\in \mathfrak{G_f},\;k,k'\in \mathfrak{K}.$ The  functions $\varrho$ are continuous, positive-definite, invariant under $g\longrightarrow kgk^\prime$ and the linear representation
$$
f\longrightarrow \varrho(f) = \int f(g)\varrho(g)dg
$$
must be a homomorphism of $A$ onto $[0,1].$\\
We next prove the existence of Helgason- functions theorem in the  fuzzy sense.

\begin{proposition}\cite{Helgason}
\normalfont
Let $f$  be a complex-valued continuous function on $\mathfrak{G_f}$, not identically $0$. Then $f$ is a spherical function if and only if
$$
\int_\mathfrak{K}f(xky)dk = f(x)f(y)
$$
\end{proposition}
\begin{theorem}
\normalfont
Let $\mathfrak{G_f}$ be a fuzzy Lie group and $\mathfrak{K}$ a closed, compact fuzzy subgroup of $\mathfrak{G_f}$. Let $f\in C_c(\mathfrak{K}\backslash \mathfrak{G_f}/\mathfrak{K})$ then $f$ is a spherical function on $\mathfrak{G_f}$, $f:\mathfrak{G_f}\longrightarrow [0,1]$ and $f$ satisfies the function
$$f:C_c(\mathfrak{K}\backslash \mathfrak{G_f}/\mathfrak{K})\longrightarrow \mathfrak{G_f}_{|_{[0,1]}}$$
such that
\begin{equation*}
\int_Kf(xky)d_k =
\left\{
\begin{split}
& f(x)f(y), ~~\mbox{  if } \exists\; \mbox{no}\; x,y\in \mathfrak{G_f}\;  \mbox{such that}\;  t_1<f(xy)<t_2\\
& 0,~~~\mbox{  if otherwise}
\end{split}
\right.
\end{equation*}
\end{theorem}
\begin{proof}
Let $f$ be a fuzzy set on a fuzzy Lie group $\mathfrak{G_f}$. Let $t\in [0,1].$ We define the $t$-cut set or $t$-level set of $\mu,$ by
$$
f_t = \{x\in \mathfrak{G_f}:f(x)\geq t\}.
$$
If $x,y\in \mathfrak{G_f}$ and $f(x) = t_1$ and $f(y) = t_2$.\\
Then $x\in f_{t_1}$ and $y\in f_{t_2}$. If $t_2<t_1$, we have
$$
f_{t_1}\subseteq f_{t_2}\Rightarrow x\in f_{t_2}
$$
and we have $x,y\in f_{t_2}$ and since $f_{t_2}$ is a subgroup of $\mathfrak{G_f}$, by definition we have $xy\in f_{t_2}$. Therefore
$$
f(xy)\geq t_2
$$
and $t_2 = \min(f(x), f(y)).$\\
Also let $x\in \mathfrak{G_f}$ and $f(x) = t$, we have $x\in f_t$, since $f_t$ is a subgroup. $x^{-1}\in f_t$. Therefore
$$
f(x^{-1})\geq t.
$$
Then $f(x^{-1})\geq f(x)$. This shows that $f$ is a subgroup of $\mathfrak{G_f}$.\\
\par Since $f$ is a homomorphism, we have
$$
f(xy) = f(x)f(y).
$$
Now\\
let $x,y\in \mathfrak{G_f}$ and $k\in \mathfrak{K}$, let $f(x)f(y)  = t_0$ and $\int_\mathfrak{K }f(xky)d_k = t_1$.\\
Then,\\
$xy\in f_{t_0}$ and $xky\in f_{t_1}$ for
$$
\int_\mathfrak{K} f(xky)d_k = f(x)f(y).
$$

Then $\exists$ no $xy\in\mathfrak{G_f}$ such that
$$
t_0<f(xy)<t_1.
$$
Let
$$
\int_Kf(xky) = f(x)f(y)
$$
and we have
$$
f_{t_1} = f_{t_0}
$$
If $\exists$ $x,y\in G$ such that, $t_0<f(xy)<t_1$ then $f_{t_1}\subset f_{t_0}$. Since $xy\in f_{t_0}$ but $xy\notin f_{t_1}$\\
then this contradicts our statement.\\
Conversely, if $\exists$ no $xy\in G$ such that
$$
t_0<f(xy)<t_1.
$$
If $t_0<t_1$ we get $f_{t_1}\subseteq f_{t_0}$. Let $xy\in f_{t_0}$, we have $f(xy)\geq t_0$ and $f(xy)\geq t_1$. This implies that $f(x,y)$ does not lie between $t_1$ and $t_0$. Hence $xy\in f_{t_1}$ and $f_{t_0}\subseteq f_{t_1}$. Hence $f_{t_0} = f_{t_1}$ and
$$
\int_K f(xky)dk = f(x)f(y).
$$
\end{proof}

\begin{center}

\end{center}

\end{document}